\documentclass[12pt]{amsart}
\usepackage{graphicx, amssymb, amsthm, amsfonts, latexsym, epsfig, amsmath}
\usepackage{nicefrac} 

\newtheorem{definition}{Definition}[section]
\newtheorem{theorem}[definition]{Theorem}

\newtheorem{claim5}{Claim}

\hyphenchar\font=-1
	
\newcommand{\diam}{\text {diam}}

\newcommand{\B}[1]{\ensuremath{\mathbb{#1}}}
\newcommand{\N}{\B{N}}

\newcommand{\Z}{\B{Z}}

\newcommand{\ba}[1]{\bar{#1}}

\begin{document}

\bibliographystyle{amsplain}

\title[Equivariant]{Equivariant Compactifications}

\author[J. Maissen]{James Maissen}
\address{Department of Mathematics, University of Texas at Brownsville, Brownsville, TX 78520, USA}
\email{jmaissen@yahoo.com}


\begin{abstract}
We show sufficient criteria for a group of homeomorphisms acting on a metric space $X$ to extend to one acting on a given compactification of $X$.
We give examples for when this can fail when one of the criteria is not met.
\end{abstract}

\maketitle

\section {Introduction}
Let $A: G \times X \rightarrow X$ be a group action on a separable metric space $X$.
The goal of this paper is to present sufficient conditions when given a compactification $C$ of $X$ that the group action $A$ will extend to a group action $\hat A: G \times C \rightarrow C$.
Obviously, these conditions must require that each element of the group $G$ have an extension from a homeomorphism on $X$ to a homeomorphism on $C$.
However, it is possible for the group action to fail to extend, as a group action, even when each of the elements of $G$ extends.
We give examples when each of the further conditions that we impose is not met that causes the group action to fail to extend.
Originally this was motivated in an attempt on the Hilbert-Smith conjecture, by means of extending a free $p$-adic group action on the space of irrationals{\cite {Maissen}}.
However this result can be extended to include complete separable metric spaces and have its own uses beyond our attack on the Hilbert-Smith conjecture.

\section{Extending Group Actions}

In the theorem below we give sufficient conditions to guarantee the extension of the group action as a group action. 
Later we will discuss possible consequences from not imposing these conditions.

\begin{theorem}
\label {groupaction}
Let $X$ be a metric space, $G$ be a compact metric group, and $A: G \times X \rightarrow X$ be a topological group action. 
If $(C,d_c)$ is a metric compactification of $X$ such that for all $g \in G$ the map $A(g, \cdot) = g(x): X \rightarrow X$ extends continuously to $\hat A(g, \cdot) = \hat g(x):C \rightarrow C$, then $\hat A: G \times C \rightarrow C$ is a continuous group action.
\end{theorem}

\begin{proof}
For ease of notation, define for each $g \in G$ the map $g: C \rightarrow C$ by $g(x) = \hat A (g,x)$.

By way of a contradiction, suppose that $\hat A$ is not continuous. 
With this assumption we will show that there is an element $h \in G$ such that the function $h :=   \hat A(h, \cdot):C\rightarrow C$ is not continuous.

Since $\hat A$ is not continuous, $C$ is metric, and $X$ is dense in $C$ there is a sequence $\{(g_i,x_i)\}_{i=1}^\infty \subset G \times X$ such that $(g_i,x_i) \rightarrow (g,x) \in G \times C$ but $g_i(x_i) \nrightarrow g(x) \in C$.

\begin{claim5}
Without loss of generality we may assume $g = e \in G$.
\end{claim5}

If the sequence $g^{-1}g_i(x_i) \rightarrow x \in C$, then $gg^{-1}g_i(x_i) \rightarrow g(x)$ since $g \in G$ implies $g$ is continuous. 
But then $g_i(x_i) \rightarrow g(x)$ since $gg^{-1}g_i(x_i)=g_i(x_i)$ as $x_i \in X$ and $A$ is a group action on $X$. 
Thus the sequence $g^{-1}g_i(x_i) \nrightarrow x \in C$ and  we may assume that $g = e \in G$.

\begin{claim5}
Without loss of generality we may assume that $g_i(x_i)$ converges to $ y \neq x$.
\end{claim5}
Since $g_i(x_i) \nrightarrow x \in C$ and $C$ is a compact metric space the sequence $\{g_i(x_i)\}_{i=1}^\infty$ has a convergent subsequence such that $g_{i_k}(x_{i_k}) \rightarrow y \in C \setminus \{x\}$. 
So we can assume that $g_i(x_i) \rightarrow y \neq x$.

\begin {claim5}
There is a sequence of group elements $\{ h_j \in G \}_{j=0}^\infty$, sets $\{W_j \subseteq X \}_{j=1}^\infty$, and points $\{ u_j \in X \}_{j=1}^\infty$, $\{ v_j \in X \}_{j=1}^\infty$ such that

\begin{enumerate} 

 \item \quad $\diam (W_j) \rightarrow 0$, 

\item \quad $u_j \in W_j \cap X$ and $v_j \in W_j \cap X$, and 

\item  \quad $h_j(u_i) \in U$ and $h_j(v_i) \in V$ for all $i \le j$.
\end{enumerate}

\end{claim5}
Let $0<r<  d_c(x,y) \neq 0$, let $n_0 = 0$, let $U := B_{\frac r 2}(x)$ and let $V := B_{\frac r 2}(y)$, so $\ba U \cap \ba V = \emptyset$.
For sake of induction, we define $h_0 := e \in G$ and $W_1 := B_{\frac r 2}(h_0^{-1}(x)) = U$.

Since $x_i \rightarrow x$ and $g_i(x_i) \rightarrow y$ in $C$, there exists a number $M_0 > 0$ such that whenever $n>M_0$ both $x_n \in U$ and $g_n(x_n) \in V$. 
Let $m_1 := M_0 +1$ and let $u_1:=h_0^{-1}(x_{m_1}) =x_{m_1} \in U=W_1$.

Since $g_i \rightarrow e$ in $G$, for the compact (finite) set $\{x_{m_1}\}$, there is a number $N_1 > m_1$ such that whenever $n>N_1$ the point $g_n(x_{m_1}) \in U$. 
Let $n_1 := N_1+1$ and let $v_1:=h_0^{-1}(x_{n_1})=x_{n_1} \in U=W_1$.

Let $h_1 := g_{n_1} \circ h_0 = g_{n_1} \in G$, then $h_1(u_1)=g_{n_1}(x_{m_1}) \in U$ and $h_1(v_1) = g_{n_1}(x_{n_1}) \in V$.
Suppose for a number $k \ge 1$ and all $1 \le j \le k$ that 

\begin{enumerate}
\item \quad $h_j \in G$ 

\item \quad $W_j = B_{\frac r j}(h_{j-1}^{-1}(x))$ 

\item \quad $u_j \in W_j \cap X$ and  $v_j \in W_j \cap X$ 

\item \quad $n_j > m_j > n_{j-1}$ 

\item \quad $h_{j-1}(u_j) = x_{m_j}$ and $h_{j-1}(v_j)=x_{n_j}$ 

\item \quad $h_j(u_i) \in U$ for all $1 \le i \le j$ and $h_j(v_i) \in V$ for all $1 \le i \le j$

\end{enumerate}
Let $W_{k+1} := B_{\frac r {k+1}}(h_{k}^{-1}(x))$.

Since $g_i \rightarrow e$ in $G$, for the compact (finite) set $\{h_k(u_i), h_k(v_i) : 1 \le i \le k\}$ there is a number $M_k > n_k$ such that whenever $n>M_k$ the point $h^{-1}_k(x_n) \in W_{k+1}$, the points $g_n(h_k(u_i)) \in U$ for all $1 \le i \le k$ and the points $g_n(h_k(v_i))\in V$ for all $1 \le i \le k$. 
Let $m_{k+1} := M_k +1$ and let $u_{k+1} := h^{-1}_k(x_{m_{k+1}}) \in W_{k+1}$.

Since $g_i \rightarrow e$ in $G$, for the compact (finite) set $\{x_{m_{k+1}}\}$ there is a number $N_{k+1} > m_{k+1}$ such that whenever $n>N_{k+1}$ the point $g_n(x_{m_{k+1}}) = g_n(h_k(u_{k+1})) \in U$. 
Let $n_{k+1} := N_{k+1} + 1$ and let $v_{k+1} := h^{-1}_k(x_{n_{k+1}})$.

Let $h_{k+1} := g_{n_{k+1}} \circ h_k \in G$ then $h_{k+1}(u_i) \in U$ and $h_{k+1}(v_i) \in V$ for $1 \le i \le k+1$.
Then as claimed there is the following for all $n \in \N$ and all $k \le n$ 
\begin{enumerate}

\item \quad $h_n \in G$  

\item \quad $W_n = B_{\frac r {n-1}}(h_{n-1}^{-1}(x)) \subseteq X$, thus $\diam(W_n) \rightarrow 0$  

\item \quad $u_n \in W_n \cap X$ and $v_n \in W_n \cap X$  

\item \quad $h_n(u_k) \in U$ and $h_n(v_k) \in V$  

\end{enumerate}

\begin {claim5}
There is an $h \in G$ such that $h(x) := \hat A(h, \cdot): C\rightarrow C$ is not continuous which contradicts our assumption that $\hat A$ is continuous.

\end{claim5}

Since $C$ is a compact metric space, there is a subsequence of $\{u_k\}_{k=1}^\infty$ which converges to a single point in $C$. 
Without loss of generality ignore subindices. Let $z \in C$ such that $u_k \rightarrow z$ in $C$. 
Since $\diam (W_n) \rightarrow 0$, we also have $v_k \rightarrow z$ in $C$ as well.

Since $G$ is a compact metric group, the sequence $\{h_k\}_{k=1}^\infty$ has a convergent subsequence to an element in $G$. 
Without loss of generality ignore subindices. Let $h \in G$ such that $h_k \rightarrow h$ in $G$.

For a fixed $i \in \N$ and the corresponding compact (finite, two point) set $\{u_i, v_i\}$ since $h_k \rightarrow h$ in $G$ it follows that $h_k(u_i) \rightarrow h(u_i)$ and $h_k(v_i) \rightarrow h(v_i)$. 
For $k > i$ the points $h_k(u_i) \in U$ and $h_k(v_i) \in V$ thus $h(u_i) \in \ba U$ and $h(v_i) \in \ba V$.

Since $h\in G$ we should have $h:C\rightarrow C$ being continuous, so $h(u_i)\rightarrow h(z) \in \ba U$ and $h(v_i) \rightarrow h(z) \in \ba V$. 
Thus we have $h(z) \in \ba U \cap \ba V = \emptyset$ which means that $h:C\rightarrow C$ is discontinuous as claimed.

This is a contradiction and thus it must be the case that $\hat A$ is continuous as desired.

\end{proof}

\section{Examples}
One might think that simply requiring every element of a group action upon a space to extend would be sufficient to force the group action to extend.
However, we impose by Theorem \ref {groupaction} two additional criteria, namely that the group be compact and that the compactification be metric.
We now give an example of a simple space, provide free group actions on that space, and demonstrate that these group actions will not extend when one of the criteria from the theorem is not met.

Consider the space  $X := \N \times \Z_2$, where $\N$ denotes the natural numbers and $\Z_2$ the group of two elements. 
Let $C$ be a compactification of $\N$ and let $Y : = C \times \Z_2$ be the corresponding compactification of $X$.

\quad

\quad

\begin{picture}(300,130)(0,0)

\put (10,130) {\makebox(0,0) { $\Z_2$ } }

\put (10, 110) {\makebox(0,0) {\tiny {1} }}

\put (35, 105) {\makebox(0,0) {$\bullet$} }
\put (80, 105) {\makebox(0,0) {$\bullet$} }
\put (125, 105) {\makebox(0,0) {$\bullet$} }
\put (170, 105) {\makebox(0,0) {$\bullet$} }
\put (215, 105) {\makebox(0,0) {$\dots$}}
\put (250, 105) {\makebox(0,0) {$\therefore$} }
\put (310, 105) {\makebox(0,0) {$C \setminus \N \times \{1\}$ } }

\put (10, 80) {\makebox(0,0) {\tiny {0} }}

\put (35, 75) {\makebox(0,0) {$\bullet$ } }
\put (80, 75) {\makebox(0,0) {$\bullet$} }
\put (125, 75) {\makebox(0,0) {$\bullet$} }
\put (170, 75) {\makebox(0,0) {$\bullet$} }
\put (215, 75) {\makebox(0,0) {$\dots$}}
\put (250, 75) {\makebox(0,0) {$\therefore$} }
\put (310, 75) {\makebox(0,0) {$C \setminus \N \times \{0\}$ } }

\put (37, 55) {\makebox(0,0) { \tiny {1} } }
\put (82, 55) {\makebox(0,0) { \tiny {2} } }
\put (127, 55) {\makebox(0,0) { \tiny {3} } }
\put (172, 55) {\makebox(0,0) { \tiny {4} } }
\put (217, 55) {\makebox(0,0) {\tiny {\dots} } }

\put (105, 25) {\makebox (0,0) { $\N $ } }
\put (300, 25) {\makebox(0,0) {$C \setminus \N$} }

\end{picture}
\quad

We will first show that the condition that the group be compact is not a spurious choice.
We will define a free group action on $X$ by a non-compact group, and show that this group action does not extend to any compactification, $Y$, defined above.
Define for each $i \in \N$, the homeomorphism $f_i : X \rightarrow X$ given by:

$$f_i (n, z) := \left \{ \begin{matrix} (n,z \oplus_{_{\Z_2}} 1) && n = i \\ (n,z) && n \neq i \end{matrix} \right \}$$

\noindent
Let $e : X \rightarrow X$ be the identity, and let $G_w$ be the group generated by $\{ f_i \}_{i=1}^{\infty}$ (in other words the countable weak product of $\Z_2$ actions).  
For any compact set $K$, let $N_K := max \{ n \in \N | (n, z) \in K \}$, then for $i > N_K$ we have ${f_i}_{|_K} = e$. 
We have that $f_i \rightarrow e$ with the topology generated by convergence on compact sets.

For each $i \in \N$ the homeomorphism $f_i$ extends to a $\hat {f_i} : Y \rightarrow Y$ by:

$$\hat {f_i} (\alpha) := \left \{ \begin{matrix} f_i(\alpha) && \alpha \in X \\ \alpha && \alpha \in Y \setminus X \end{matrix} \right \}$$

\noindent
Let $\hat e: Y \rightarrow Y$ be the identity on $Y$. 
The extension, $\hat f_i$, is also a homeomorphism since the function $f_i$ is a homeomorphism, the composition $\hat f_i \circ \hat f_i = \hat e$, and for any neighborhood $N_y$ of $Y \setminus X$ where $N_y \subset Y \setminus \{(n,z) | n \le i, z \in \Z_2 \}$ we have $\hat {f_i}_{|_{N_y}} = \hat e$ and, thus $\hat f_i$ is continuous.
Since every $g \in G_w$ is the finite composition of elements from $\{ f_i \}_{i=1}^{\infty}$, the map $g : X \rightarrow X$ extends to a homeomorphism $\hat g : Y \rightarrow Y$ as well.

Yet, in the sequence $\{ \hat {f_i} \}_{i=1}^{\infty}$, we do not have $\hat {f_i} \rightarrow \hat e$. 
To see this, consider $\alpha \in Y \setminus X$ and any neighborhood, $N_\alpha \subset Y$, of $\alpha$. 
For every $N \in \N$, there is an $n_N > N$ such that $(n_N, z) \in N_\alpha$ for some $z \in \Z_2$. 
Now $\hat f_{n_N} (n_N, z) \ne \hat e (n_N, z) = (n_N, z)$, so $\lim \hat {f_i} \ne \hat e$. 

Even though every element of the group $G_w$ extends to a homeomorphism from $Y$ to itself and the group structure on $G_w$ is maintained, the topology of the group action is not.
It is possible to define extensions of $G_w$ to some compactifications of $X$, but not those compactifications in the form that we defined for $Y$.
Given a $Y$ defined above, define $Z := \nicefrac Y \sim$ by $(c, z) \sim (c^\prime, z^\prime)$ if and only if $c = c^\prime \in C \setminus \N$.
The group action $G_w$ extends as a group action to $Z$, where for every $f \in G_w$ we have $f_{| Z \setminus X} \equiv 1_Z$, the identity on the space $Z$.

Next we will justify our decision to require also that the compactification be a metric compactification.
Let us consider a larger group, a compact zero-dimensional group $G_s \cong \Pi_{i=1}^{\infty} \Z_2$. 
For every $\gamma = (\gamma_i) \in \Pi_{i=1}^{\infty} \Z_2$, define $f_\gamma : X \rightarrow X$ by $f_\gamma (n, z) :=  (n,z \oplus_{_{\Z_2}} \gamma_n)$. 
These functions yield a $G_s$ group action on $X$.
If, for every $\gamma \in G_s$, the homeomorphism $f_\gamma$ extends to a homeomorphism $\hat f_\gamma : Y \rightarrow Y$, then we will show that $C = \beta \N$, the Stone-\u{C}ech compactification of $\N$.

For each $\gamma = (\gamma_i) \in \Pi_{i=1}^{\infty} \Z_2$, define a function $F_\gamma: \N \rightarrow \{0,1\}$ by $F_\gamma (n) = \gamma_n$. 
Since $C \times \{0\}$ is open in $Y$ and since $f_\gamma$ extends to a homeomorphism $\hat f_\gamma$, there is a continuous extension $\hat F_\gamma : C \rightarrow \{0,1\}$ given by

$$ \hat F_\gamma (\alpha) = \left \{ \begin{matrix} F_\gamma (\alpha) && \alpha \in \N \\  0 && \alpha \in C \setminus \N, \hat f_\gamma (\alpha, 0) \in C \times \{0\} \\ 1 && \alpha \in C \setminus \N, \hat f_\gamma (\alpha, 0) \in C \times \{1\} \end{matrix} \right \}$$

\noindent
Since every function $F : \N \rightarrow \{0,1\}$ extends continuously to a function $\hat F: C \rightarrow \{0,1\}$ by definition of the Stone-\u{C}ech compactification, we have that the compactification $C = \beta \N$.
Since $Y = C \times \Z_2$, we have $Y \cong \beta \N$.
Since the only compact sets of $\beta \N$ that are metrizable are finite sets, the orbits of a compact group action on it would be trivial.
Consider the extension $\hat f_\gamma : Y \rightarrow Y$ of $f_\gamma$ where $\gamma = (1)_{i=1}^{\infty}$ which would perforce have to swap $(C \setminus \N) \times \{ 0 \}$ with $(C \setminus \N) \times \{ 1 \}$.
Thus $G_s$ does not act on $Y$ as a group action.

Again, as was the case with $G_w$ acting on $X$, there are extensions of the group action $G_s$ to compactifications of $X$ but we cannot mandate that those compactifications be of the form given to $Y$.
We see in this example a free compact zero dimensional group action that does not have any extension to a free action. 

Moreover the use of $Z_2$ as a factor of $X$ in our example could have been replaced by a more interesting compact group such as the $p$-adic numbers or some other pro-finite group without difficulty.

\end{document}